\documentclass[a4paper]{article}
\usepackage{bbm}
\usepackage{amsfonts}
\usepackage{amssymb}
\usepackage{amsmath}
\usepackage{amsthm}
\usepackage{graphics}
\usepackage{subfigure}
\usepackage{graphicx}
\usepackage[colorlinks,
            linkcolor=red,
            anchorcolor=blue,
            citecolor=green
            ]{hyperref}
\usepackage{url}
\usepackage{dsfont}
\title{Chen-Gackstatter type surfaces in $\mathbb{R}_{1}^{4}$: \\ deformation, symmetry, and embeddedness}
\author{Zhenxiao  Xie   \and  Xiang Ma }
\newtheorem{theorem}{Theorem}[section]
\newtheorem{proposition}[theorem]{Proposition}
\newtheorem{definition}[theorem]{Definition}

\newtheorem{lemma}[theorem]{Lemma}

\theoremstyle{remark}
\newtheorem{remark}[theorem]{Remark}

\begin{document}
\maketitle
\begin{abstract}
We find a 2-parameter family of deformations in $\mathbb{R}^4_1$
of the classical Chen-Gackstatter surface explicitly,
and show the existence of a larger 4-parameter family of deformations.
Each of them still has genus one, a unique end, with total Gaussian curvature $-\int K=8\pi$.
On the other hand, a uniqueness theorem is obtained when we assume that the surface has more than $4$ symmetries.
The problem of embeddedness is also discussed with some partial results.\\
\end{abstract}
\tableofcontents
\section{Introduction}
\label{intro}
\indent
The existence and uniqueness problems of complete
minimal surfaces in $\mathbb{R}^{3}$ under global assumptions
have always been an important problem.
One example is Lopez-Ros theorem \cite{Lopez-Ros},
which says that a complete embedded minimal surface in $\mathbb{R}^{3}$ of genus zero with finite total curvature
must be a plane or a catenoid.
Obviously, embeddedness puts strong restriction on the
global behavior of minimal surfaces in $\mathbb{R}^{3}$.
Thanks to the work of Colding, Minicozzi, Meeks, Rosenberg
and other people, we now have a much better understanding
on related problems (see the survey in \cite{Meeks-Perez}).

In the 4-dimensional Lorentz space, the situation changes dramatically.
By the generalized Weierstrass representation formula
for stationary surfaces (i.e., zero mean curvature
spacelike surfaces) in $\mathbb{R}_{1}^{4}$,
in \cite{Liu-Ma-Wang-Wang-2012} we constructed the
generalized $k$-noids, some stationary graphs over (domains of)
$\mathbb{R}^{2}$, as well as the generalized Enneper surfaces,
which could all be embedded in $\mathbb{R}_{1}^{4}$
(with genus zero and finite total Gaussian curvature).
These examples refute the Lopez-Ros theorem, and also
violates the conclusions of Schoen's famous characterization of
the catenoid \cite{Schoen} as well as Meeks and Rosenberg's
uniqueness theorem about the helicoid \cite{Meeks-Rosenberg}.
In particular, all previous uniqueness and non-existence results
in $\mathbb{R}^{3}$ need to be re-examined in this new context.

In this paper we start to consider complete, algebraic, stationary surfaces of genus one in $\mathbb{R}_{1}^{4}$.

Recall that the first of such examples in $\mathbb{R}^3$ is
the Chen-Gackstatter surface \cite{Chen-Gackstatter-1981}
which has a unique Enneper-type end (hence not embedded)
and least total curvature $-\int K\mathrm{d}M=8\pi$.
According to Lopez's uniqueness result \cite{Lopez-1992},
this is the only minimal surface with genus one and
total curvature $8\pi$, thus there does not exist
embedded examples in this class.

In contrast to the $\mathbb{R}^{3}$ case, the deformed
Enneper surface in $\mathbb{R}_{1}^{4}$ (as well as the
deformed Enneper end) could be embedded.
Thus we suspect that there might exist a deformation of
Chen-Gackstatter surface in $\mathbb{R}_{1}^{4}$ which
preserves all other properties yet avoids self-intersection.
This simple idea is the starting point of our exploration.
In particular, we give partial answers to the following questions:\\

{\bf Question 1 (uniqueness problem):} \emph{Does there exist other complete, immersed,
algebraic stationary surfaces in $\mathbb{R}_{1}^{4}$
defined on a torus with one end and total curvature
$-\int_{M} K \mathrm{d}M=8\pi$?}\\

{\bf Question 2 (embeddedness problem):} \emph{Among all possible new examples,
is there any one embedded in $\mathbb{R}_{1}^{4}$?}\\

The reader might think, in such a higher codimensional case,
the problem of embeddedness is somewhat trivial.
This view is justified when the ambient space is $\mathbb{R}^4$,
which contains a lot of embedded minimal tori with one end and total curvature $8\pi$ (see the end of the final section).
But in $\mathbb{R}^4_1$, the one more dimension corresponds
to a \emph{timelike} direction, which still put restrictions on
a stationary surface when one tries to extend it to a complete
\emph{spacelike} surface.

It is easy to observe that for any minimal surface
in $\mathbb{R}^3$ whose flux vector is horizontal,
there is a standard way of deformation
(\emph{the Lorentz deformation}) of this surface
into $\mathbb{R}_{1}^{4}$ with the same topology and is still
completely immersed. Applying this to the Chen-Gackstatter
surface in $\mathbb{R}^3$, we establish\\

{\bf Theorem A (see Theorem~\ref{thm-deform} and Proposition~\ref{prop-deform}).}
\emph{There exists a (real) 2-parameter family of deformations into
$\mathbb{R}_{1}^{4}$ of the Chen-Gackstatter surface.
Each of them still has genus one, a unique end,
with total curvature $-\int K\mathrm{d}M=8\pi$, and intersect with itself
at two points in general. Each of them shares the same
symmetry group $D_4$ as the classical one in $\mathbb{R}^3$.}\\

Besides this simple generalization, we show that
there is a larger family of deformations.\\

{\bf Theorem B (Theorem~\ref{thm-exist}).}
\emph{There exists a (real) 4-parameter family of deformations of
the Chen-Gackstatter surface which still has genus one,
a unique end, with total curvature $-8\pi$, and
the Gauss maps $\phi,\psi$ both have order $1$ at the end.}\\

To prove this existence result, we first write out the
possible Weierstrass data under the topological assumptions.
It involves essentially six independent complex parameters.
On the other hand, the two generating cycles of the torus $T^2$
put $8$ real period conditions on them in $\mathbb{R}^4_1$.
These period conditions can be put together to form
\emph{the period mapping} from the parameter space
(a domain in $\mathbb{C}^6$) to $\mathbb{R}^8$.
Since there is already the classical Chen-gackstatter surface in $\mathbb{R}^3$
satisfying the period conditions and regularity condition,
by showing that the zero vector $\vec{0}\in\mathbb{R}^8$ is
a regular value of the period mapping and applying the
preimage theorem, we show that when the Weierstrass data
was perturbed slightly in this $6$-dimensional parameter space,
there exists a 4-parameter family among them which solves
the period conditions in $\mathbb{R}^4_1$.
The regularity condition is also verified easily when
the perturbation is small enough.

So far we do not know how to write out this 4-parameter family
in Theorem~B explicitly, or whether there is any of them
embedded in $\mathbb{R}^4_1$. But under the further assumption
of a large symmetry group, we can obtain a uniqueness result as below.
Note that there has been many uniqueness and classification results
in the research of minimal and maximal surfaces when the symmetry group
is large(for recent example
we just mention \cite{Meeks+Wolf, Fujimori+Lopez}).\\

{\bf Theorem C (Theorem~\ref{thm-unique}).} \emph{Among Chen-Gackstatter surfaces of genus one, with a unique end and total curvature $-\int K\mathrm{d}M=8\pi$, the deformations given in Theorem~A are the only ones with a large symmetry group $G$ with order $|G|>4$.}\\

We would like to compare with the methods used to obtain
uniqueness result in $\mathbb{R}^3$.
The proof there needs to consider two distinct cases as below.

Case 1: the Gauss map has order 1 at the end;

Case 2: the Gauss map has order 2 at the end.

In each case one needs to write out the Weierstrass
data for the candidates explicitly which depend on several parameters,
and then to examine the period conditions.
The original proof in \cite{Lopez-1992} (also by
Bloss in 1989 in his PhD thesis)
were long and tedious, providing little insight
why such a uniqueness result should be true.

Later in \cite{Weber-2002}, Weber
developed the theory of period quotient maps and gave a simple, conceptual proof in Case~1.

On the other hand, Kusner used a simple argument to exclude Case~2 above, using the fact that the intersection multiplicity should be less than the (sum of) multiplicities (at $\infty$) of all the ends. (We call it \emph{the multiplicity inequality}.) The proof can be seen in a survey article \cite{Lopez-Martin-1999}. It is indeed a direct consequence of the monotonicity formula in $\mathbb{R}^3$.\\

In our proof to Theorem~C, besides the basic knowledge on symmetries of a torus with various modules, we also used Weber's ideas and results in \cite{Weber-2002} on period quotient maps in Case~1.

As to Case~2 in $\mathbb{R}^4_1$, unfortunately
the monotonicity formula in $\mathbb{R}^3$ could not be directly
generalized to our pseudo-Riemannian ambient space.
So we can not prove the multiplicity inequality in $\mathbb{R}_{1}^{4}$ as before. By Kusner's idea we only get a weaker result:\\

{\bf Theorem D (Proposition~\ref{thm-triple}).} \emph{Any complete, immersed, algebraic stationary surface in $\mathbb{R}^4_1$ defined on a torus with one end and total curvature $-\int_{M} K \mathrm{d}M=8\pi$ is not embedded
if the Gauss maps $\phi,\psi$ have order $2$ at the end.}\\

We organize this paper as below.

In Section~2 we briefly review the basic facts about
stationary surfaces in $\mathbb{R}_{1}^{4}$.

In Section~3, as a preparation for discussions of period conditions, we recall the notion of period quotient maps and useful notations and lemmas from Weber's paper
\cite{Weber-2002}. Here the periods are given by certain
elliptic integrals, and the conventions on signs are fixed.

Section~4 describes the Chen-Gackstatter surface as well as
its Lorentz deformation in $\mathbb{R}_{1}^{4}$. In particular
we show that they still share the same symmetry group (the dihedral group $D_4$ of order $8$). This establishes Theorem~A.

In Section~5, to deal with the existence and uniqueness of Chen-Gackstatter type surfaces, according to the orders of Gauss maps
$\{\phi,\psi\}$ at the end being $1$ or $2$,
we divide the discussion of possible examples into two cases.
In each case, based on the analysis of the divisors we write
out $\{\phi,\psi,\mathrm{d}h\}$ explicitly with some parameters
to be determined by the period conditions and the regularity condition $\phi\ne\bar\psi$.

The existence of the 4-parameter deformations (Theorem~B)
will be proved in Section~6. Based on the information given in Case~1,
we compute the Jacobian of
the period mapping and evaluate it at the specific parameter
combination which describes the Chen-Gackstatter surface
in $\mathbb{R}^3$. To show the rank of the Jacobian matrix is $8$, we need to evaluate
certain elliptic integrals given in Section~4.
The conclusion follows from this computation and the preimage theorem.

The uniqueness result (Theorem~C) is obtained in
Section~7. The assumption on symmetries forces
the underlying conformal structure to be equivalent to
either the square torus or to the equilateral torus,
and put strong restrictions on the divisors of $\phi,\psi,\mathrm{d}h$.
This greatly simplifies the discussion both in Case~1 and Case~2.

In the final Section~8, we prove that in Case~2 where Gauss maps
$\{\phi,\psi\}$ having order $2$ at the end, any possible example
always has a triple point (Theorem~D).\\

\textbf{Acknowledgement} This work is funded by the Project 10901006 of National Natural Science Foundation of China.
The authors would like to thank Professor Matthias Weber for helpful discussions on period quotient maps as well as the picture of Chen-Gackstatter surface in $\mathbb{R}^3$ created by him.
Professor Houhong Fan informed us the reference \cite{Eujalance} on symmetric tori. We thank Professor Changping Wang for consistent encouragement.

\section{Basic theory on stationary surfaces in $\mathbb{R}^4_1$}\label{sec-theory}

In this section we review the general theory about
oriented complete algebraic stationary surfaces
of finite total curvature established in \cite{Liu-Ma-Wang-Wang-2012}.
There will be no new results in this part.

First of all, stationary surfaces in $\mathbb{R}^4_1$ can
still be studied using complex functions on Riemann surfaces.
We now have two Gauss maps $\phi,\psi$ taking value in $S^2\backsimeq \mathbb{C}\cup\{\infty\}$. The generalized Weierstrass representation is the main tool we will use.

\begin{theorem}\label{thm-weierstrass}
Given holomorphic 1-form $\mathrm{d}h$ (called the height differential) and meromorphic functions
$\phi,\psi:M\rightarrow C\cup\{\infty\}$ (the Gauss maps)
globally defined on a Riemann surface $M$.
Suppose they satisfy the regularity condition
1),2) and period condition 3) as below:\\
\indent 1) $\phi\neq\overline{\psi}$ on $M$ and their poles do not coincide;\\
\indent 2) The zeros of $\mathrm{d}h$ coincide with the poles of $\phi \; or \;
 \psi$,with the same order;\\
\indent 3) Along any closed path the periods satisfy
\begin{align}
&\text{the horizontal period condition}:&~~~~&\oint_\gamma \phi \mathrm{d}h=-\overline{\oint_\gamma \psi
\mathrm{d}h}; \label{1}\\
&\text{the vertical period condition}:&~~~~\mathrm{Re}&\oint_\gamma \mathrm{d}h=\mathrm{Re}\oint_\gamma \phi\psi \mathrm{d}h=0. \label{2}
\end{align}
Then the conformal immersion given by
\begin{equation}\label{eq-weierstrass}
X=2\mathrm{Re} \int \Big(\phi+\psi, -i(\phi-\psi),
1-\phi\psi,1+\phi\psi\Big)\mathrm{d}h
\end{equation}
is a stationary surface $X:M\rightarrow \mathbb{R}_1^4$.
The induced metric is
\begin{equation}\label{eq-metric}
\mathrm{d}s^2=|\phi-\bar\psi|^2\cdot|\mathrm{d}h|^2.
\end{equation}
\indent Conversely, any stationary surface
$X:M\rightarrow \mathbb{R}_1^4$ can be represented as above
in terms of such $\phi,\psi,\mathrm{d}h $ over a Riemann surface $M$.
\end{theorem}

\begin{remark}\label{rem-congruent}
As pointed out in \cite{Liu-Ma-Wang-Wang-2012}, a Lorentz rotation in $\mathbb{R}^4_1$ will transform the Gauss maps $\phi,\psi$ and the height differential $\mathrm{d}h$ as below:
\begin{equation}\label{eq-congruent}
\phi \rightarrow  \frac{a\phi+b}{c\phi+d},~
\psi \rightarrow \frac{\bar{a}\psi+\bar{b}}{\bar{c}\psi+\bar{d}}, ~\mathrm{d}h \rightarrow
(c\phi+d)(\bar{c}\psi+\bar{d})\mathrm{d}h.\quad
\begin{pmatrix}a & b\\c & d \end{pmatrix} \in SL(2,\mathbb{C})
\end{equation}
Conversely, the action \eqref{eq-congruent} yields congruent
stationary surfaces in $\mathbb{R}^4_1$. This formula is frequently used to simplify our discussions.

We also note that a dilation in $\mathbb{R}^4_1$ will only change the height differential by multiplying with a real factor, i.e., $\mathrm{d}h\to a\mathrm{d}h$, without influence on $\phi,\psi$.
\end{remark}

Compared to minimal surfaces in $\mathbb{R}^3$, stationary surfaces in $\mathbb{R}^4_1$ are more complicated. One difficulty is that there
exist examples with essential singularities in the Weierstrass data yet the total curvature is still finite. The other one
comes from the so-called \emph{singular ends}, where
the limits of the two light-like normal directions coincide (i.e., $\phi=\bar\psi$ at such an end; otherwise the end will be called a \emph{regular end}).

Both phenomena are
revealed and discussed in \cite{Liu-Ma-Wang-Wang-2012}.
But in this paper we will only consider surfaces without
such kinds of singularities. Thus the Gauss-Bonnet type formulae given there could be simplified as below.

\begin{theorem}\label{thm-gaussbonnet}
Consider a compact Riemann surface punctured at several points (called the ends)
and denote it as $M \simeq \overline{M}-\{p_1,\cdots,p_r\}$.
Let $X: M \rightarrow \mathbb{R}^4_1$ be a complete stationary surface given by \eqref{eq-weierstrass} with
meromorphic Weierstrass data $\phi,\psi,\mathrm{d}h$ on $\overline{M}$ satisfying the regularity conditions and period conditions.
(This is called an \emph{algebraic stationary surface}).
We assume that $\phi\ne\bar\psi$ on the whole $\overline{M}$.
Then the Jorge-Meeks formula is still valid in our case:
\begin{eqnarray}
-\int_{M} K dM &=&4\pi \deg\phi ~=~4\pi \deg\psi \label{eq-deg}\\
&=&2\pi\big(2g-2+r+\sum_{j=1}^r d_j\big). \label{eq-jorgemeeks}
\end{eqnarray}
Here the left hand side is called the total (Gaussian) curvature; it is related with other topological quantities: $g$ is the genus of $\overline{M}$, $r$ is the number of ends, and $d_j+1$ is the order of poles of $X_z\mathrm{d}z$ at the end $p_j$.
\end{theorem}

\section{Elliptic integrals and period quotient maps}
\label{sec-period}

As a preparation to construct examples and to discuss existence
problems on genus 1 surfaces, we review the basic ideas and results in
\cite{Weber-2002} on period quotient maps in this section.
We follow the conventions in \cite{Weber-2002}
with some modification on notations.

It is well-known that any Riemann surface of genus 1
can be realized as a regular algebraic curve in $\mathbb{C}P^2$ with a standard form
\[
T_{\lambda} =\{[x,y,1]\in \mathbb{C}P^2 \mid y^2=x(x-1)(x-\lambda)\},~~~~
\lambda\in\hat{\mathbb{C}}-\{0,1,\infty\}.
\]
This is called an elliptic curve with modular invariant $\lambda$, with a holomorphic 1-form
\[
\omega=\mathrm{d}z=\mathrm{d}x/y,
\]
whose integration on $T_{\lambda}$ recovers
the universal covering $\mathbb{C}$ with a global coordinate $z$.
Choose two cycles $\gamma_1,\gamma_2$ on $T_{\lambda}$  generating $H_1(T_{\lambda},\mathbb{Z})$ and define
periods $\omega_j=\int_{\gamma_j}\omega, j=1,2.$
They generate a lattice $\Lambda(\omega_1,\omega_2)\subset\mathbb{C}$, with
\[
T_{\lambda}\simeq \mathbb{C}/\Lambda(\omega_1,\omega_2)
\simeq \mathbb{C}/\Lambda(1,\tau).
\]
The period quotient
\[
\tau=\tau(\lambda)=\omega_2(\lambda)/\omega_1(\lambda)
\]
is the lattice constant which describes exactly the shape of
the fundamental parallelogram of $T_{\lambda}$ and
distinguishes tori with different conformal types.\\

More generally, we consider a family of meromorphic 1-forms
\begin{equation*}
\alpha(r,s,t;\lambda)=x^r(x-1)^s(x-\lambda)^t\frac{\mathrm{d}x}{y}
\end{equation*}
on $T_{\lambda}$. Denote the periods by
\begin{equation*}
\alpha_j(r,s,t;\lambda)=\int_{\gamma_j}\alpha(r,s,t;\lambda),~~~~
(j=1,2)
\end{equation*}
and the period quotient map by
\begin{equation*}
\sigma(r,s,t;\lambda)
=\alpha_2(r,s,t;\lambda)/\alpha_1(r,s,t;\lambda).
\end{equation*}
In this paper, most of the time we are interested in the following 1-forms:
\begin{equation}\label{eq-Phi}
\omega(\lambda)= \frac{\mathrm{d}x}{y}=\alpha(0,0,0;\lambda), \quad
\Phi(\lambda)= \frac{x\mathrm{d}x}{y}=\alpha(1,0,0;\lambda).
\end{equation}
Hence their periods are denoted by
\begin{equation}\label{eq-Phij}
\omega_j(\lambda) =\oint_{\gamma_j} \omega(\lambda), \quad
\Phi_j(\lambda) =\oint_{\gamma_j} \Phi(\lambda),~~~~(j=1,2).
\end{equation}
For given real parameters $r,s,t$, these are generally
multi-valued functions on the modular sphere
$\lambda \in \hat{\mathbb{C}}-\{0, 1, \infty\}$.
But as in \cite{Weber-2002}, they will be regarded as
single-valued functions on the upper half plane.
Next one chooses $\gamma_1,\gamma_2$ as in \cite{Weber-2002}.
In particular, the orientation of the cycles is fixed by the convention that when $\lambda=-1$,
\begin{equation}\label{eq-Phisign}
\omega_1(-1)=\int_{\gamma_1}\frac{\mathrm{d}x}{y}\in\mathbb{R}^+, \quad
\omega_2(-1)=\int_{\gamma_2}\frac{x\mathrm{d}x}{y}\in \mathrm{i}\mathbb{R}^+.
\end{equation}
These period quotient maps are therefore well-defined in a consistent way by these conventions.

Extending classical results, Weber proved in \cite{Weber-2002} that
the period quotient map $\lambda \mapsto \sigma(r, s, t; \lambda)$
is a Riemann mapping from the upper half plane to a
circular triangle (with angles $|r+t|\pi,|s+t|\pi,|r+s|\pi$
at the images of the points $0,1,\infty$).
These geometric pictures help us to understand the behavior of
the period integrals clearly, and to solve uniqueness/existence
problems in minimal surface theory in an elegant and simple way.
For such applications, one key result we will use later is
as below.
\begin{theorem}[Theorem~2.22 in \cite{Weber-2002}]\label{thm-period}
$\lambda=-1$ is the only solution to the period-quotient equation
\begin{equation*}
\overline{\sigma(1,0,0;\lambda)}=\sigma(0,1,1;\lambda).
\end{equation*}
Here $\sigma(1,0,0;\lambda)$ and $\sigma(0,1,1;\lambda)$ are
the period quotient of $x\mathrm{d}x/y$,
and $(x-1)(x-\lambda)\mathrm{d}x/y=y\mathrm{d}x/x$, respectively.
\end{theorem}

We collect some classical facts as below which will be used later. This is indeed part of the content of Lemma~2.1 in \cite{Weber-2002}. The third fact may be derived directly, or
see the more general Proposition~2.16 in \cite{Weber-2002}.
\begin{proposition}\label{prop-legendre}
On the elliptic curve $T_{\lambda}:y^2=x(x-1)(x-\lambda)$, we have
\begin{enumerate}
\item $x(z)$
has a unique pole and a unique zero, both of order $2$, on $T_{\lambda}$.
\item The Legendre relation: $\omega_1\Phi_2-\omega_2\Phi_1=8\pi \mathrm{i}$.
\item $\frac{x^2\mathrm{d}x}{y}\thickapprox \frac{2(\lambda+1)}{3}\Phi-\frac{\lambda}{3}\omega$. ($\thickapprox$ means both sides are equal up to an exact 1-form.)
\end{enumerate}
\end{proposition}

\section{Chen-Gackstatter surface and its Lorentz deformation in
$\mathbb{R}_1^4$}\label{sec-deform}

We begin this section by recalling the construction of the
Chen-Gackstatter surface in $\mathbb{R}^3$ and its basic properties (see \cite{Weber-2001} for a reference).
Using a very general method of deformation in $\mathbb{R}_1^4$,
we obtain a $2$-parameter family
of Chen-Gackstatter type stationary surfaces.
The symmetry property is similar.
Yet the self-intersection pattern changes: any of them
has exactly two self-intersection points in $\mathbb{R}_1^4$.

In the previous section we have given the Weierstrass
representation formula in $\mathbb{R}_1^4$ depending on the data $\phi,\psi,\mathrm{d}h$. When $\phi\psi\equiv -1$, we may rewrite
\[\phi=G,~~\psi=-1/G,\]
and we get the classical Weierstrass representation for minimal surfaces in $\mathbb{R}^3$.

In particular, the Chen-Gackstatter surface in $\mathbb{R}^3$
is represented on the elliptic curve
\[y^2=x(x-1)(x+1),\]
with the Gauss map $G$ and height differential $\mathrm{d}h$ given by
\begin{equation}\label{eq-CG}
G=\frac{1}{\rho}\cdot\frac{x}{y}, \qquad \mathrm{d}h=\mathrm{d}x.
\end{equation}
The real parameter $\rho$ can be explicitly determined as
\begin{equation}\label{eq-rho}
\rho=\frac{\sqrt{6}\Gamma(\frac{3}{4})}{\Gamma(\frac{1}{4})}
\approx 0.8279
\end{equation}
to satisfy the horizontal period condition.
The height differential $\mathrm{d}h$ is an exact $1$-form,
hence satisfying the vertical period condition in any case.
Below we summarize basic facts about this surface.

\begin{enumerate}
\item Divisors. Note that the underlying Riemann surface is a square torus punctured at one point. The divisors of $G$ and
$\mathrm{d}h$ are pictured in Figure~1 below, where the end is
located at the corner(s) of the square.
\begin{figure}[htbp]
\centering
\subfigure{
\label{G}
\includegraphics[width=0.4\textwidth]{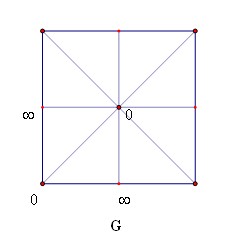}}
\subfigure{
\label{dh}
\includegraphics[width=0.4\textwidth]{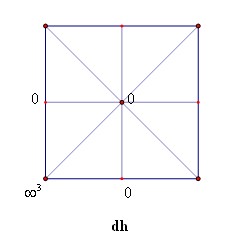}}
\caption{Divisors of the Weierstrass data}
\label{Gdh}
\end{figure}
\item Global picture.
This is a complete minimal
surface in $\mathbb{R}^3$ of finite total curvature $-\int K\mathrm{d}M=8\pi$ and genus $1$.
Its geometric behavior is shown in
Figure~2 below. (This picture is obtained from the on-line Minimal Surface Archive \cite{Weber-archive}.)
\begin{figure}[htbp]
\centering
\includegraphics[width=.7\textwidth]{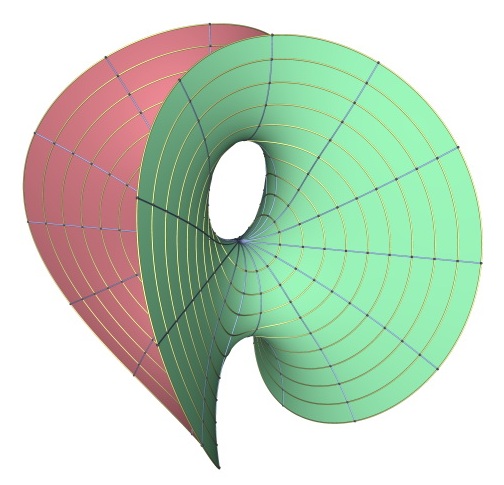}
\caption{The Chen-Gackstatter surface (courtesy of M. Weber)}
\end{figure}
\item Symmetry.
From the picture we see that it has
the same symmetry as the Enneper surface, which are generated
by the reflections with respect to the $OX_1X_3$ plane and
the $OX_2X_3$ plane, and two $180$-degree rotational symmetry
around the straight lines $\{X_1=\pm X_2, X_3=0\}$.
In particular the symmetry group is $D_4$ including $8$ elements.
This could be verified easily using \eqref{eq-CG} and
the Weierstrass representation formula, together with
the symmetry properties of the elliptic functions $x,y$
defined on curve $y^2=x(x-1)(x+1)$ as well as on the fundamental domain:
\begin{itemize}
   \item Under the reflection to the horizontal or the vertical lines passing through the center of the square, $x\to \bar{x},~y\to \pm \bar{y}.$
   \item Under the reflection to the two diagonals, $x\to -\bar{x},~y\to \pm \mathrm{i}\bar{y}.$
\end{itemize}
\item
Self-intersection. It has an Enneper end of multiplicity $3$;
thus it is not embedded in $\mathbb{R}^3$).
From Figure~2 we can see that along the self-intersection line
on this surface, there are only two self-intersection points
located on the vertical $X_3$-axis.
\end{enumerate}

In \cite{Liu-Ma-Wang-Wang-2012}, when we constructed
the generalized Jorge-Meeks k-noid in Example~6.6,
we used a special deformation
depending on two complex parameters $a, b$ satisfying
$a^2-b^2=1$ and $|a|>|b|$. If we assume that
$X_z dz=(\Theta_1, \Theta_2, \Theta_3)$ is the holomorphic differential of a
minimal surface $X$ in $\mathbb{R}^3$, then this deformation
$\widetilde{X}$ in $\mathbb{R}^4_1$ is given by
\begin{equation}\label{eq-deform}
\widetilde{X}_z dz=(\Theta_1, \Theta_2, a \Theta_3, b \Theta_3)
=(\Theta_1, \Theta_2, \frac{\zeta+\zeta^{-1}}{2}\Theta_3, \frac{\zeta-\zeta^{-1}}{2} \Theta_3)
\end{equation}
where $a= \frac{\zeta+\zeta^{-1}}{2},b= \frac{\zeta-\zeta^{-1}}{2}$ is the general solution to
$a^2-b^2=1$, and $\zeta=a+b$.

\begin{definition}\label{def-deform}
The stationary surface $\widetilde{X}$ given by \eqref{eq-deform}
is called \emph{the Lorentz deformation} in $\mathbb{R}^4_1$
of the original minimal surface $X: M \longrightarrow \mathbb{R}^3$.
It depends on a complex parameter
 $\zeta\in \mathbb{C}-\mathrm{i}\mathbb{R}$.
When $X$ has Gauss map $G$
and height differential $\mathrm{d}h$, the Weierstrass data
of $\widetilde{X}$ is given by
\begin{equation}\label{eq-deform2}
\phi=\frac{G}{\zeta}~, \quad
\psi=\frac{-1}{\zeta G}~,\quad
\widetilde{\mathrm{d}h}=\zeta \mathrm{d}h~.
\end{equation}
\end{definition}

\begin{theorem}\label{thm-deform}
Suppose $X:M\to \mathbb{R}^3$ is a complete immersed
minimal surface without real or imaginary vertical periods
(i.e., $\int_\gamma \mathrm{d}h=0$
along any closed curve $\gamma$ on $M$).
Suppose $\zeta\in \mathbb{C}\backslash\{0\}$ is not purely imaginary.
Then $\widetilde{X}$ given above by \eqref{eq-deform} is
a (real) $2$-parameter family of complete stationary surfaces fully immersed in $\mathbb{R}^4_1$ (i.e., any of them is not contained in any $3$-dimensional affine subspace) with regular ends.

Conversely, any fully immersed stationary surface
$\widetilde{X}:M\to\mathbb{R}^4_1$ satisfying
$\phi\psi=-1/\zeta^2$ for some constant $\zeta$
is a Lorentz deformation of
an immersed minimal surface $X:M\to\mathbb{R}^3$.
\end{theorem}

\begin{remark}\label{rem-1parameter}
When parameter $\zeta$ takes two different values $\zeta_1,\zeta_2$
which are real linear dependent (i.e., $\zeta_2/\zeta_1=t\in\mathbb{R}^+$), the corresponding deformations
are congruent to each other under a Lorentz transformation
in $\mathbb{R}^4_1$. One needs only to notice Remark~\ref{rem-congruent} and take
$\left(\begin{smallmatrix}a &b\\c &d \end{smallmatrix}\right)=\left(\begin{smallmatrix}\sqrt{t} &0\\0 &1/\sqrt{t} \end{smallmatrix}\right)$ in
\eqref{eq-congruent}.
So there is only a $S^1$-family of non-congruent
deformations depending on the argument $\mathrm{Arg}(\zeta)$.
For convenience, sometimes we may take
\[
\zeta=\mathrm{e}^{\mathrm{i}\theta},\quad
a=\cos\theta,\quad b=\mathrm{i}\sin\theta
\]
without loss of generality.
In that case it is related to the original minimal surface
$X=(X_1,X_2,X_3):M\to\mathbb{R}^3$ by
\begin{equation}\label{eq-deform3}
\widetilde{X}_\theta=(X_1,X_2,\cos\theta X_3,-\sin\theta X_3^*),
\end{equation}
where $X_3^*$ is the harmonic conjugate of $X_3$.
\end{remark}

\begin{remark}\label{rem-3space}
For a stationary $\widetilde{X}:M\to\mathbb{R}^4_1$,
it is contained in some $3$-dimensional affine subspace of
$\mathbb{R}^4_1$ if, and only if, $\phi\psi$ is a real constant.
In particular, when this constant is a negative number,
this surface is congruent to the original $X$.
(up to a Lorentz transformation in $\mathbb{R}^4_1$).
Together with the converse part of the theorem above, we see
that stationary surfaces with $\phi\psi\equiv c\in \mathbb{C}$ are essentially
coming from classical examples in $3$-space.
\end{remark}

\begin{proof}[Proof to Theorem~\ref{thm-deform}]
The period conditions are almost trivial due to \eqref{eq-deform3}. We need only to add that
the harmonic conjugate $X_3^*$ of $X_3$ is a well-defined function by the vanishing of vertical periods.

Next we verify that the Weierstrass data given in \eqref{eq-deform}
satisfy the regularity conditions 1),2) as in Theorem~\ref{thm-weierstrass}.
The second one is trivial since the divisors
are the same as the original $X$ in $\mathbb{R}^3$.
On the other hand, at one point $p$,
\[
\phi(p)=\overline{\psi(p)}~ \Leftrightarrow~
\frac{G(p)}{\zeta}=\frac{-1}{\overline{\zeta G(p)}}
~ \Leftrightarrow~ |G(p)|=1,\zeta\in \mathrm{i}\mathbb{R}.
\]
So when $\zeta\notin \mathrm{i}\mathbb{R}$,
the regularity condition (1) holds true at any point. In particular the end is a regular end.
(Note that by Fujimoto's famous theorem \cite{Fujimoto}, the image of $G$
can not omit more than $4$ points on $\mathbb{C}\cup\{\infty\}$
when the original minimal surface $X:M\to\mathbb{R}^3$ is complete,
and there always exists a point $p$ with $|G(p)|=1$.
Thus $\zeta\notin \mathrm{i}\mathbb{R}$ is also necessary
to guarantee regularity.)

By the formula \eqref{eq-metric}, the deformation $\widetilde{X}$ has metric
\[
\mathrm{d}s^2=|\phi-\bar\psi|^2\cdot|\widetilde{\mathrm{d}h}|^2
=\left|G-\frac{\zeta}{\bar\zeta \bar{G}}\right|^2\cdot|\mathrm{d}h|^2
\]
Since $\zeta\notin \mathrm{i}\mathbb{R}$, the completeness of
the metric $|G-1/\bar{G}|^2|\mathrm{d}h|^2$ implies that
$\widetilde{X}\to\mathbb{R}^4_1$ is also complete. This finishes
the proof of the first part.

To show the converse, suppose $\widetilde{X}$ is
fully immersed in $\mathbb{R}^4_1$ with Weierstrass data
$\{\phi,\psi,\widetilde{\mathrm{d}h}\}$ and
$\phi\psi=-1/\zeta^2$ is a constant, which is not a real number
by Remark~\ref{rem-3space}.
The vertical period condition
\[
\mathrm{Re}\oint_\gamma\widetilde{\mathrm{d}h}=0,\quad \mathrm{Re}\oint_\gamma\phi\psi\widetilde{\mathrm{d}h}
=\mathrm{Re}\left(\frac{-1}{\zeta^2}\oint_\gamma \widetilde{\mathrm{d}h}\right)=0
\]
for any closed path $\gamma$ implies that $\widetilde{\mathrm{d}h}$ is exact. Then the desired $X\to\mathbb{R}^3$ is described by
\[
G=\zeta\phi, \quad \mathrm{d}h=\frac{1}{\zeta}\widetilde{\mathrm{d}h}.
\]
It is obvious that $\widetilde{X}$ is a Lorentz deformation of
 $X:M\to\mathbb{R}^3$.
\end{proof}

\begin{remark}
The Alias-Palmer deformation given in \cite{Alias} (see also Example~6.5 in \cite{Liu-Ma-Wang-Wang-2012})
also satisfies $\phi\psi=-a$ for a complex constant $a$.
Yet it differs from our Lorentz deformation.
In particular, it requires that the original minimal surface $X$
in $\mathbb{R}^3$ has neither real nor imaginary periods.
Compared to \eqref{eq-deform2}, the formula for the Alias-Palmer deformation can be given explicitly as:
\[
\phi=G,\psi=-a/G, \widetilde{\mathrm{d}h}=\mathrm{d}h.
\]
Upon closer investigation we find that it
is obtained from $X$ after three steps. First we take some suitable $X'$ in the associated family of $X$, and scale it,
which only changes the height differential
$\mathrm{d}h\to \sqrt{a}\mathrm{d}h$. The second step is applying the Lopez-Ros deformation \cite{Lopez-Ros, Meeks-Perez, Lopez-Martin-1999} and a rotation in the $OX_1X_2$ plane, which only changes the Gauss map
$G\to \sqrt{a}G$. By the vanishing of periods, we still get
a complete minimal surface in $\mathbb{R}^3$ of the same topology type. Finally we use the Lorentz deformation with $\zeta=\sqrt{a}$, and the result is equivalent to the Alias-Palmer deformation.
\end{remark}

The next proposition describes the geometry of the deformed Chen-Gackstatter surface $\widetilde{X}$.
\begin{proposition}\label{prop-deform}
The Lorentz deformation $\widetilde{X}$ of the Chen-Gackstatter surface $X$ is a real $2$-parameter family of
complete stationary surfaces in $\mathbb{R}_1^4$ defined on a torus with
a regular end, and the total curvature is $-\int K \mathrm{d}M=8\pi$. Moreover:

(1) $\widetilde{X}$ has two self-intersection points when
$\zeta$ is not a real number.

(2) $\widetilde{X}$ has $8$ symmetries like the original Chen-Gackstatter surface $X$.
\end{proposition}
\begin{proof}
We prove the two listed properties. Other ones are quite straightforward to see.
In terms of the previous section, the underlying Riemann surface
$M$ is equivalent to a square torus $\overline{M}$ punctured at one point. Here
\[
\overline{M}\cong\{[x,y,1] \in \mathbb{C}P^2:y^2=x(x-1)(x+1)\}\cong \mathbb{C}/\Lambda(1,\mathrm{i}),
\]
where $\Lambda(1,\mathrm{i})$ denotes the lattice group generated by $\{1,\mathrm{i}\}$.
Let $z$ be a global coordinate of $\mathbb{C}$,
then $x=x(z)$ is something like a Weierstrass-$\wp$ function.
Conversely, given the elliptic curve $y^2=x(x-1)(x+1)$,
$z$ is given by the holomorphic $1$-form $\mathrm{d}z=\mathrm{d}x/y$.
We will use coordinate $z$ to describe points on $M$.

To find out the symmetries of $\widetilde{X}$,
without loss of generality we assume
$\zeta=\mathrm{e}^{\mathrm{i}\theta}$ according to Remark~\ref{rem-1parameter}. $\zeta$ is not real or purely imaginary by our assumptions. The formula
\eqref{eq-deform3} now takes the form
\begin{equation}\label{eq-deform4}
\widetilde{X}=(X_1,X_2,4\cos\theta \mathrm{Re}(x),4\sin\theta\mathrm{Im}(x)),\quad \cos\theta\sin\theta\ne 0.
\end{equation}
It is easy to see that the symmetries of $x,y$
(see Fact~3 about the Chen-Gackstatter surface)
extend to symmetries of the ambient space $\mathbb{R}^4_1$.
This verifies the second conclusion.

At any self-intersection point with
\[
\widetilde{X}(z')=\widetilde{X}(z''),\quad z'\ne z''(\mathrm{mod}~\Lambda),
\]
by comparing the third and fourth components in \eqref{eq-deform4}
we know $x(z')=x(z'')$. The equation $y^2=x(x^2-1)$ implies $y(z')=\pm y(z'')$.
Yet $z',z''$ correspond to distinct points on the torus;
this forces $y(z')=-y(z'')$. Such a pair $(z',z'')$ could
be chosen in the fundamental domain containing $z=0$, and
by the symmetry property of $x,y$, we get $z''=-z'$.
The symmetry properties of the Chen-Gackstatter surface
show that the horizontal components
\[
X_1(z'')=X_1(-z')=-X_1(z'), X_2(z'')=X_2(-z')=-X_2(z').
\]
(Or use \eqref{eq-CG} to see that both $\int G\mathrm{d}h$
and $\int \mathrm{d}h/G$ are odd functions).
Thus $\widetilde{X}(z')=\widetilde{X}(-z')$ forces
\[X_1(z')=X_2(z')=0=X_1(z'')=X_2(z'').\]
So this corresponds exactly to the self-intersection points
of $X$ in $\mathbb{R}^3$ on the $X_3$-axis.
By Fact~4 about the Chen-Gackstatter surface,
we know there are exactly two of such self-intersections.
\end{proof}

Observe that the Lorentz deformation makes the Enneper end embedded in $\mathbb{R}^4_1$, yet the whole surface still fails to be embedded. Since Enneper surface could be deformed into an
embedded surface in $\mathbb{R}^4_1$, we are wondering whether
it is possible to do similar things for Chen-Gackstatter type surfaces.
This motivates the discussion in the rest of this paper.

\section{Weierstrass data of Chen-Gackstatter type surfaces in $\mathbb{R}^4_1$}
\label{sec-divisor}

To answer Question~1 and 2 in the introduction,
we will describe the possible Weierstrass data
for Chen-Gackstatter type surfaces in this section.
Such a complete immersed stationary surface
$X:M\rightarrow \mathbb{R}^4_1$ is assumed to be of algebraic type,
with total curvature $-\int _M K \mathrm{d}M=8\pi$,
defined on $M \cong T^2 - \{P\}$, a torus with one regular end
(i.e.,$\phi\ne\bar\psi$ at $P$). (The exceptional case with a good singular end will not be considered in this paper.)

According to Remark~\ref{rem-congruent}, without loss of generality, we can assume that at the end,
$\phi(P)=0,~ \psi(P)=\infty,~ \mathrm{d}h(P)=\infty$. Denote
\[
\phi(P)=0^l,\psi(P)=\infty^m,\mathrm{d}h(P)=\infty^n,
\]
where $l,m,n\in\mathbb{Z}^+$ are the corresponding multiplicities.

By Theorem~\ref{thm-gaussbonnet} and our assumptions, we know
$\deg\phi=\deg\psi=2$ and the multiplicity at the unique end $d=3$. Thus $1\le l,m\le 2$.

The highest order pole of $X_z\mathrm{d}z$ comes from $\psi\mathrm{d}h$, thus $m+n=4$.

The regularity condition implies that
$\mathrm{d}h,\phi\mathrm{d}h,\psi\mathrm{d}h,\phi\psi\mathrm{d}h$
are holomorphic on $T^2-\{P\}$.
Hence their poles, if exist, must be located at the end $P$.
Since there exist no meromorphic functions of $\deg=1$ on $T^2$,
we conclude that neither of the $1$-forms above can have a simple pole at $P$. That means $n\ne 1, n-l\ne 1, m+n-l\ne 1$.

We conclude that, either $l=1, m=1, n=3$, or $l=m=n=2$.
Using the regularity condition that the zeros of
$\mathrm{d}h$ coincide with the poles of $\phi$ or $\psi$
with the same order, their divisors are derived as in the next proposition.
\begin{proposition}\label{prop-divisor}
There are two possibilities for the divisors of Weierstrass data $\{\phi,\psi,\mathrm{d}h\}$ on $T^2$ when we have normalized $\phi(P)=0,~ \psi(P)=\infty$:
$$
\text{Case}~1:(\phi)=P+C_1-V_1-V_2, (\psi)=-P-C_2+N_1+N_2,
(\mathrm{d}h)=-3P+V_1+V_2+C_2;
$$
$$
\!\!\!\!\!\!\!\!\!\!\text{Case}~2:(\phi)=2P-V_1-V_2, (\psi)=-2P+C_1+C_2, (\mathrm{d}h)=-2P+V_1+V_2.
$$
Here $C_1, C_2, V_1, V_2, N_1, N_2$ are points on $T^2-\{P\}$
\end{proposition}

In the following we discuss these two cases separately.

\subsection{The Weierstrass data in Case~1}
\label{sub-case1}
Consider complete, oriented regular algebraic stationary surface defined on a torus whose Weierstrass data
have the following divisors:
\begin{equation}
\label{11}
(\phi)\!=\!P\!+\!C_1\!-\!V_1\!-\!V_2,
(\psi)\!=\!-\!P\!-\!C_2\!+\!N_1\!+\!N_2, (\mathrm{d}h)\!=\!-3P\!+\!V_1\!+\!V_2\!+\!C_2.~
\end{equation}
It follows
\begin{align}
&(\phi \mathrm{d}h)=-2P+C_1+C_2, \label{12}\\
&(\psi \mathrm{d}h)=-4P+V_1+V_2+N_1+N_2, \label{13}\\
&(\phi \psi \mathrm{d}h)=-3P+C_1+N_1+N_2. \label{14}
\end{align}

Given any global coordinate $u$ on the universal cover of $T^2$, we write out
\[
\phi \mathrm{d}h=f(u)\mathrm{d}u.
\]
It is a standard fact that the elliptic function $f(u)$ of degree $2$ and its derivative $f'(u)$ must satisfy a cubic equation,
which yields an embedding into $\mathbb{C}P^2$:
\[
T^2=\{[f(u),f'(u),1]\in \mathbb{C}P^2\mid f'(u)^2=\varrho(f(u)-e_1)(f(u)-e_2)(f(u)-e_3)\}
\]
where $\varrho$ is a complex constant,
$\{\infty,e_1,e_2,e_3\}$ are the four distinct branch points of $f(u)$.
Make a change of variables as below:
\[
x=\frac{f(u)-e_1}{e_2-e_1},~ y=\frac{f'(u)}{\sqrt{\varrho (e_2-e_1)^3}},~ z=\sqrt{\varrho(e_2-e_1)}\cdot u.
\]
Then the algebraic curve is in its standard form
\begin{align}
T^2=\{[x,y,1]\in \mathbb{C}P^2\mid y^2=x(x-1)(x-\lambda)\}, \label{35}
\end{align}
where $\lambda=(e_3-e_1)/(e_2-e_1)$ is the modular invariant
of this $T^2$, and the new complex coordinate $z$ satisfies
$\mathrm{d}z=\mathrm{d}x/y$ as in Section~\ref{sec-period}.

Next we express Weierstrass data $\{\phi,\psi,\mathrm{d}h\}$ by
elliptic functions $x$ and $y$. The 3rd-order pole of either $\mathrm{d}h$
or $\phi\psi\mathrm{d}h$ can always be canceled by a suitable
combination of $y\mathrm{d}z$ and $x\mathrm{d}z$, after which
there leaves only with a holomorphic $1$-form.
(This comes from the fact that these 1-forms have a unique pole at $P$,
and the residue at $P$ must vanish by the residue theorem.
Note that this also guarantees that the periods along a small
circle around $P$ always vanish.) Thus we can write
\begin{equation}\label{eq-x0}
\mathrm{d}h=(by+cx+e)\mathrm{d}z,~~
\phi\psi\mathrm{d}h=(ly+mx+n)\mathrm{d}z,
\end{equation}
where $b,c,e,l,m,n$ are complex parameters which can be chosen
arbitrarily except the restriction $b\ne 0, l\ne 0$. On the other hand, we can rewrite
\[
\phi\mathrm{d}h=f(u)\mathrm{d}u=a(x-x_0)\mathrm{d}z,
~~~a=\sqrt{\frac{e_2-e_1}{\varrho}},~x_0=\frac{-e_1}{e_2-e_1}.
\]
Compared with the divisor information given at the beginning of
this subsection, we see
\[
P=[0,1,0],~C_1=[x_0,-y_0,1],~C_2=[x_0,y_0,1],
\]
where $(x_0,\pm y_0)$ are two points on $y^2=x(x-1)(x-\lambda)$.
Because $C_1,C_2$ are the zeros of $\phi\psi\mathrm{d}h$ and $\mathrm{d}h$ separately, there should be
$e=-by_0-cx_0, n=ly_0-mx_0$.

Finally, note that according to Remark~\ref{rem-congruent},
we can always make a Lorentz transformation in $\mathbb{R}^4_1$
and a dilation so that the divisors of $\{\phi,\psi,\mathrm{d}h\}$ are
preserved, but the parameter $a$ is normalized to be $1$.
Thus without loss of generality we will assume $a=1$.
This normalization also fixes our immersion in $\mathbb{R}^4_1$
up to translations.

In summary, we can write out the Weierstrass data explicitly as below depending on 6 complex parameters $\{b,c,l,m,\lambda\}$ and $x_0$ (or $y_0$):
\begin{align}
\phi&=\frac{x-x_0}{b(y-y_0)+c(x-x_0)}, \label{15}\\
\psi&=\frac{l(y+y_0)+m(x-x_0)}{x-x_0}, \label{16}\\
\mathrm{d}h&=[b(y-y_0)+c(x-x_0)]\mathrm{d}z \thickapprox cx\mathrm{d}z-(cx_0+by_0)\mathrm{d}z, \label{17}\\
\phi \psi \mathrm{d}h&=[l(y+y_0)+m(x-x_0)]\mathrm{d}z \thickapprox mx\mathrm{d}z+(-mx_0+ly_0)\mathrm{d}z, \label{18}\\
\phi\mathrm{d}h&=(x-x_0) \mathrm{d}z, \label{19}\\
\psi\mathrm{d}h&=\frac{[l(y+y_0)+m(x-x_0)][b(y-y_0)+c(x-x_0)]}
{x-x_0} \mathrm{d}z. \label{20}
\end{align}
Here the symbol $\thickapprox$ means that both sides are equal up to adding an exact 1-form. In particular, by $y^2-y_0^2=(x-x_0)[x^2+xx_0+x_0^2-(\lambda+1)(x+x_0)+\lambda]$ and the third conclusion in Proposition~\ref{prop-legendre}, we know
\[
\psi\mathrm{d}h\thickapprox Ax\mathrm{d}z+B\mathrm{d}z,
\]
\begin{equation}\label{1A}
A=bl\left[x_0-\frac{\lambda+1}{3}\right]+cm,
\end{equation}
\begin{equation}\label{1B}
B=(cl-bm)y_0+bl\left[x_0^2-(\lambda+1)x_0+\frac{2\lambda}{3}\right]
-cmx_0.
\end{equation}
The data $\{\phi, \psi, \mathrm{d}h\}$ given in \eqref{15}$\sim$\eqref{20} represents a complete
algebraic stationary surface if, and only if, they satisfy
the regularity condition $\phi\neq\bar{\psi}$ and the period conditions.

Here the regularity condition amounts to
\begin{equation}\label{1reg}
|x-x_0|^2 \neq \overline{[b(y-y_0)+c(x-x_0) ]}[l(y+y_0)+m(x-x_0)],
\end{equation}
besides $[x_0,\pm y_0,1]$ ($=C_1,C_2$).

To write down the period conditions explicitly, recall
the notations introduced in Section~\ref{sec-period}:
\[
\omega=\mathrm{d}z,~\Phi=x\mathrm{d}z;\quad \text{periods:} ~\omega_i=\oint_{\gamma_i}\omega,~
\Phi_i=\oint_{\gamma_i}\Phi,
\]
where $\gamma_i (i=1,2)$ are two cycles generating
the first homology group of $T^2$. We have
\begin{align}
&\mathrm{Re}\oint_{\gamma_{i}} \mathrm{d}h=0 &~ \Longleftrightarrow \quad &\mathrm{Re}[c\Phi_i-(cx_0+by_0)\omega_i]=0, \label{period11}\\
&\mathrm{Re}\oint_{\gamma_{i}} \phi \psi \mathrm{d}h=0 &~ \Longleftrightarrow \quad &\mathrm{Re}[m\Phi_i-(mx_0-ly_0)\omega_i]=0, \label{period12}\\
&\overline{\oint_{\gamma_{i}}\phi \mathrm{d}h} = -\oint_{\gamma_{i}}\psi \mathrm{d}h &~ \Longleftrightarrow\quad &-\overline{\Phi_i}+\overline{x_0}\overline{\omega_i}=A\Phi_i+B\omega_i. \label{period13}
\end{align}
In general it is quite difficult to solve the period problem completely.
An exceptional case is when $x_0=0$; we will show that
the only solutions are the Lorentz deformations described in
Section~\ref{sec-deform}. This result will also be
used in Section~\ref{sec-unique} to obtain another uniqueness
theorem.

\begin{proposition}\label{prop-unique}
Suppose $X:T^2-\{P\} \to \mathbb{R}_1^4$ is a
complete regular algebraic stationary surface of
total curvature $-\int_M K \mathrm{d}M=8\pi$
defined on the elliptic curve $y^2=x(x-1)(x-\lambda)$,
whose Gauss map $\{\phi,\psi\}$ have order 1 at the end $P$.
It has Weierstrass data as given by \eqref{15}$\sim$\eqref{20}. Furthermore, if
\[x_0=0,\]
then $\lambda=-1$ (square torus), and this surface must be
a Lorentz deformation of the classical Chen-Gackstatter surface.
\end{proposition}
\begin{proof}
When $x_0=0, y_0=0$, the period conditions (\ref{period11})$\sim$(\ref{period13}) are simplified to
\begin{align}
&Re[c\Phi_i]=0, \label{period14}\\
&Re[m\Phi_i]=0, \label{period15}\\
&-\overline{\Phi_i}=\left[mc-\frac{bl}{3}(\lambda+1)\right]\Phi_i
+\frac{2bl}{3}\lambda \omega_i. \label{period16}
\end{align}
The Weierstrass data $\{\phi,\psi,\mathrm{d}h\}$ takes
the simple form as below:
\begin{equation}\label{01}
\phi=\frac{ax}{by+cx}~, \quad
\psi=\frac{ly+mx}{ax}~, \quad
\mathrm{d}h=(by+cx)\mathrm{d}z=(by+cx)\frac{\mathrm{d}x}{y}~.
\end{equation}
\indent We will first show that $c=m=0$. Otherwise, suppose $c \neq 0$. Multiply with $\bar{c}$ on both sides of \eqref{period16} and invoking \eqref{period11}, we obtain
\begin{equation}\label{02}
\left[mc-\frac{bl}{3}(\lambda+1)\right]\bar{c}\Phi_i
+\frac{2bl}{3}\lambda\cdot \bar{c}
\omega_i=-\bar{c}\overline{\Phi_i}= c\Phi_i,
\end{equation}
The Legendre relation (see Proposition~\ref{prop-legendre})
implies that the 2-vector $(\omega_1,\omega_2)$ are linearly independent to
$(\Phi_1,\Phi_2)$. So the coefficients of $\omega_i,\Phi_i$
in \eqref{02} must cancel with each other, i.e.,
\[
\left[mc-\frac{bl}{3}(\lambda+1)\right]\bar{c}-c=0, \quad
\frac{2bl}{3}\lambda=0.
\]
But the second equality contradicts with our assumption
$\lambda \ne 0, b\ne 0, l\ne 0$. For the same reason we know $m=0$.\\
\indent
We have shown that $\{\phi,\psi,\mathrm{d}h\}$ must take the following form:
\begin{equation}\label{03}
\phi=\frac{x}{by}, \quad \psi=\frac{ly}{x}, \quad \mathrm{d}h=b\mathrm{d}x.
\end{equation}
Using the notation in Section~\ref{sec-period},
\[
\oint_{\gamma_{i}}\phi \mathrm{d}h = \oint_{\gamma_{i}}\frac{x}{y}\mathrm{d}x =
\alpha_i(1,0,0;\lambda),
\]
\[\oint_{\gamma_{i}}\psi \mathrm{d}h = bl\oint_{\gamma_{i}}\frac{(x-1)(x-\lambda)}{y}\mathrm{d}x = bl\cdot\alpha_i(0,1,1;\lambda)
\]
The horizontal period condition implies $\overline{\sigma(1,0,0;\lambda)}=\sigma(0,1,1;\lambda)$.
Then by Weber's Theorem~\ref{thm-period} we know $\lambda=-1$.

Weierstrass data in \eqref{03} verifies $\phi\psi=l/b$, which is
a constant. By the second part of Theorem~\ref{thm-deform}, we know this surface must be the Lorentz deformation of a Chen-Gackstatter type surface in $\mathbb{R}^3$. The latter one
can only be the classical example. This completes our proof.
\end{proof}

\subsection{The Weierstrass data in Case 2}
\label{sub-case2}
In the second case mentioned in Proposition~\ref{prop-divisor}, the divisors of the Weierstrass data $\{\phi, \psi, \mathrm{d}h\}$ are
\begin{equation}\label{21}
(\phi)=2P-V_1-V_2, ~(\psi)=-2P+N_1+N_2, ~ (\mathrm{d}h)=-2P+V_1+V_2.
\end{equation}
Compared to Case~1, here we use $\mathrm{d}h$
instead of $\phi\mathrm{d}h$ to give the parametrization of
this torus in almost the same way:
\[
T^2=\big\{[x,y,1]\in \mathbb{C}P^2 \big| y^2=x(x-1)(x-\lambda)\big\}.
\]
Then we can still use explicit formulas to express the
Weierstrass data in terms of $x,y$ and holomorphic 1-form $\mathrm{d}z=\mathrm{d}x/y$:
\begin{align}
&\mathrm{d}h=a(x-x_0)\mathrm{d}z, \quad \phi \mathrm{d}h=e\mathrm{d}z, \quad \phi \psi \mathrm{d}h=(bx+c)\mathrm{d}z, \label{22}\\
&\Rightarrow ~~~\psi \mathrm{d}h=\frac{a}{e}(bx+c)(x-x_0)\mathrm{d}z, \label{23}\\
&\Rightarrow ~~~\phi=\frac{e}{a(x-x_0)}~, \quad \psi=\frac{bx+c}{e}, \label{24}
\end{align}
where $a,b,c,e,x_0$ are all complex parameters with $a,b,e \neq 0$.

By Remark~\ref{rem-congruent}, we can still make the
normalization $e=1$ without loss of generality.
This condition, together with another normalization $|a|=1$,
fix our immersion in $\mathbb{R}^4_1$ up to translations.

The period conditions can be written down as below:
\begin{align}
&\mathrm{Re}\oint_{\gamma_{i}} \mathrm{d}h=0  & \Longleftrightarrow \quad &\mathrm{Re}[a\Phi_i-ax_0\omega_i]=0, \label{25}\\
&\mathrm{Re}\oint_{\gamma_{i}} \phi \psi \mathrm{d}h=0  & \Longleftrightarrow \quad &\mathrm{Re}[b\Phi_i+c\omega_i]=0, \label{26}\\
&\overline{\oint_{\gamma_{i}}\phi \mathrm{d}h} = -\oint_{\gamma_{i}}\psi \mathrm{d}h & \Longleftrightarrow \quad &-\frac{1}{a}\overline{\omega_i}
=\left[\frac{2b}{3}(\lambda+1)+c-bx_0\right]\Phi_i -
\left[\frac{b\lambda}{3}+cx_0\right]\omega_i. \label{27}
\end{align}
Note that we use $\frac{x^2\mathrm{d}x}{y}\thickapprox \frac{2(\lambda+1)}{3}\Phi-\frac{\lambda}{3}\omega$ once again to obtain the last equation.

\section{A 4-parameter family of deformations}
\label{sec-exist}

In this section, we will prove the existence of 4-parameter family of deformations of the classical Chen-Gackstatter surface. They include
the Lorentz deformations mentioned in Proposition~\ref{prop-deform} and
Proposition~\ref{prop-unique} as a subfamily. In particular this shows
that the uniqueness result in Proposition~\ref{prop-unique} is
generally not true when $x_0\ne 0$.

Consider the integration of $X_z\mathrm{d}z$ using the Weierstrass data \eqref{15}$\sim$\eqref{20} in Case~1 and the representation formula \eqref{eq-weierstrass} along $\gamma_1,\gamma_2$. The result is a vector $\mathcal{P}$ consisting of the $8$ real periods, depending on six complex parameters
\[
v=(b,c,l,m,y_0,\lambda)\in \mathbb{C}^6,~~b,l\ne 0;~\lambda\ne 0,1.
\]
The period conditions amount to say that $\mathcal{P}=\vec{0}\in\mathbb{R}^8$.

As well-known it is difficult to solve these period conditions directly. Alternatively, we view $\mathcal{P}$ as a mapping from $v\in\mathbb{C}^6=\mathbb{R}^{12}$ to $\mathbb{R}^8$, called the \emph{period mapping}, and apply the preimage theorem.
Note that the set of preimages of $\vec{0}$ is non-empty, since when
\[
v^*=(b,c,l,m,y_0,\lambda)=(\rho,0,-\rho,0,0,-1),
~~\rho=\frac{\sqrt{6}\Gamma(\frac{3}{4})}{\Gamma(\frac{1}{4})}
\approx 0.8279,
\]
the period condition is satisfied ($\mathcal{P}(v^*)=\vec{0}$) and
the example is the classical Chen-Gackstatter surface.
If only we can show that the Jacobian matrix of this period mapping at $v^*$ is non-degenerate, the preimage theorem will guarantee the desired
existence result.

Write out the real vector
\[
v=(b_1,b_2,c_1,c_2,l_1,l_2,m_1,m_2,y^1_0,y^2_0,\lambda_1,\lambda_2)
\in \mathbb{R}^{12},
\]
where $b=b_1+\mathrm{i}b_2,\cdots,\lambda=\lambda_1+\mathrm{i}\lambda_2$.
According to Subsection~\ref{sub-case1},
the period mapping
\[
\mathcal{P}:\mathbb{R}^{12}\to \mathbb{R}^{8}; ~~v\mapsto \mathcal{P}(v),
\]
is written explicitly as
\[
\mathcal{P}(v)=
\begin{pmatrix}
\mathrm{Re}[c\Phi_1-(cx_0+by_0)\omega_1]\\
\mathrm{Re}[c\Phi_2-(cx_0+by_0)\omega_2]\\
\mathrm{Re}[m\Phi_1-(mx_0-ly_0)\omega_1]\\
\mathrm{Re}[m\Phi_2-(mx_0-ly_0)\omega_2]\\
\mathrm{Re}[A\Phi_1+B\omega_1+\overline{\Phi_1}-\overline{x_0 \omega_1}]\\
\mathrm{Re}[A\Phi_2+B\omega_2+\overline{\Phi_2}-\overline{x_0 \omega_2}]\\
\mathrm{Im}[A\Phi_1+B\omega_1+\overline{\Phi_1}-\overline{x_0 \omega_1}]\\
\mathrm{Im}[A\Phi_2+B\omega_2+\overline{\Phi_2}-\overline{x_0 \omega_2}]
\end{pmatrix}
\]
where we recall that
\begin{align*}
&\omega_i=\oint_{\gamma_i} \mathrm{d}x/y,
~~~\Phi_i=\oint_{\gamma_i} x\mathrm{d}x/y;\\
&A ~=~ bl\left(-\frac{\lambda+1}{3}+x_0\right)+cm, \\
&B ~=~ y_0(cl-bm)+bl\left(x_0^2-(\lambda+1)x_0
+\frac{2\lambda}{3}\right)-cm x_0.
\end{align*}

Consider the property of $\mathcal{P}$ near $v^*=(\rho,0,-\rho,0,0,-1)$. Observe that $x_0$ is an implicit (differentiable) function $x_0=x_0(\lambda, y_0)$
near $(x_0,y_0,\lambda)=(0,0,-1)$ using the equation
\begin{equation}\label{61}
f(x_0,y_0,\lambda)= y_0^2-x_0(x_0-1)(x_0-\lambda)=0.
\end{equation}
This follows from $\frac{\partial f}{\partial x_0}(0,0,-1)=1\ne 0.$
Moreover, $\frac{\partial f}{\partial y}=0=\frac{\partial f}{\partial \lambda}$ at $(0,0,-1)$. So
\begin{equation}\label{62}
\frac{\partial x_0}{\partial \lambda}=\frac{\partial x_0}{\partial y_0}=0,~~~\text{when}~ y_0=0, \lambda=-1.
\end{equation}

It is straightforward (yet a little bit tedious) to compute out the partial derivatives of $\mathcal{P}(v)$ at $v^*$. Denote the
corresponding components as $v_i, \mathcal{P}_j, 1\le i\le 12, 1\le j\le 8.$ The Jacobian matrix is
\begin{equation*}
\left(\frac{\partial\mathcal{P}_j}{\partial v_i}\right)_{v^*}=2
\begin{pmatrix}\begin{smallmatrix}
0&0&0&0&\frac{2\rho}{3}\mathrm{Re}\omega_1
&\frac{2\rho}{3}\mathrm{Re}\omega_2
&\frac{2\rho}{3}\mathrm{Im}\omega_1
&\frac{2\rho}{3}\mathrm{Im}\omega_2\\
0&0&0&0&\frac{-2\rho}{3}\mathrm{Im}\omega_1
&\frac{-2\rho}{3}\mathrm{Im}\omega_2
&\frac{2\rho}{3}\mathrm{Re}\omega_1
&\frac{2\rho}{3}\mathrm{Re}\omega_2\\
\mathrm{Re}\Phi_1&\mathrm{Re}\Phi_2&0&0&0&0&0&0\\
-\mathrm{Im}\Phi_1&-\mathrm{Im}\Phi_2&0&0&0&0&0&0\\
0&0&0&0&\frac{-2\rho}{3}\mathrm{Re}\omega_1
&\frac{-2\rho}{3}\mathrm{Re}\omega_2
&\frac{-2\rho}{3}\mathrm{Im}\omega_1
&\frac{-2\rho}{3}\mathrm{Im}\omega_2\\
0&0&0&0&\frac{2\rho}{3}\mathrm{Im}\omega_1
&\frac{2\rho}{3}\mathrm{Im}\omega_2
&\frac{-2\rho}{3}\mathrm{Re}\omega_1
&\frac{-2\rho}{3}\mathrm{Re}\omega_2\\
0&0&\mathrm{Re}\Phi_1&\mathrm{Re}\Phi_2&0&0&0&0\\
0&0&-\mathrm{Im}\Phi_1&-\mathrm{Im}\Phi_2&0&0&0&0\\
-\rho \mathrm{Re}\omega_1&-\rho \mathrm{Re}\omega_2
&-\rho \mathrm{Re}\omega_1&-\rho \mathrm{Re}\omega_2&0&0&0&0\\
\rho \mathrm{Im}\omega_1&\rho \mathrm{Im}\omega_2
&\rho \mathrm{Im}\omega_1&\rho \mathrm{Im}\omega_2&0&0&0&0\\
0&0&0&0&C_1+e_1&C_2+e_2&-D_1-f_1&-D_2-f_2\\
0&0&0&0&D_1-f_1&D_2-f_2&C_1-e_1&C_2-e_2
\end{smallmatrix}\end{pmatrix}~.
\end{equation*}
Here
\begin{equation}\label{-0}
\begin{aligned}
&C_1=\frac{\rho^2}{3}\mathrm{Re}\Phi_1
-\frac{2\rho^2}{3}\mathrm{Re}\omega_1
+\frac{2\rho^2}{3}\mathrm{Re}
\frac{\partial \omega_1}{\partial \lambda}~, ~
&&e_1=\mathrm{Re}\frac{\partial \Phi_1}{\partial \lambda}~,\\
&C_2=\frac{\rho^2}{3}\mathrm{Re}\Phi_2
-\frac{2\rho^2}{3}\mathrm{Re}\omega_2
+\frac{2\rho^2}{3}\mathrm{Re}
\frac{\partial \omega_2}{\partial \lambda}~, ~
&&e_2=\mathrm{Re}\frac{\partial \Phi_2}{\partial \lambda},\\
&D_1=-\frac{\rho^2}{3}\mathrm{Im}\Phi_1
+\frac{2\rho^2}{3}\mathrm{Im}\omega_1
-\frac{2\rho^2}{3}\mathrm{Im}
\frac{\partial \omega_1}{\partial \lambda}~, ~
&&f_1=\mathrm{Im}\frac{\partial \Phi_1}{\partial \lambda}~,\\
&D_2=-\frac{\rho^2}{3}\mathrm{Im}\Phi_2
+\frac{2\rho^2}{3}\mathrm{Im}\omega_2
-\frac{2\rho^2}{3}\mathrm{Im}
\frac{\partial \omega_2}{\partial \lambda}~, ~
&&f_2=\mathrm{Im}\frac{\partial \Phi_2}{\partial \lambda}~.
\end{aligned}
\end{equation}
These $\omega_i, \Phi_i, \frac{\partial \omega_i}{\partial \lambda},
\frac{\partial \Phi_i}{\partial \lambda}$
take their values at $\lambda=-1$.
To evaluate them, according to the conventions
in Section~\ref{sec-period}, when $\lambda <0$
one can write out explicitly
\begin{equation}\label{-1}
\begin{aligned}
&\omega_1(\lambda)=2\int_{\lambda}^{0}
\sqrt{\frac{1}{x(x-1)(x-\lambda)}}\mathrm{d}x
=2 \int_{0}^{1}\sqrt{\frac{1}{t(1-t)(1-t\lambda)}}\mathrm{d}t, \\
&\Phi_1(\lambda)=2\int_{\lambda}^{0}
\sqrt{\frac{x}{(x-1)(x-\lambda)}}\mathrm{d}x
=2\lambda\int_{0}^{1}\sqrt{\frac{t}{(1-t)(1-t\lambda)}}\mathrm{d}t,\\
&\omega_2(\lambda)=2\mathrm{i}
\int_{0}^{1}\sqrt{\frac{1}{x(1-x)(x-\lambda)}}\mathrm{d}x,   \\ &\Phi_2(\lambda)=2\mathrm{i}
\int_{0}^{1}\sqrt{\frac{x}{(1-x)(x-\lambda)}}\mathrm{d}x.
\end{aligned}
\end{equation}
Here we evaluate for real variables, and the square roots
always take non-negative values, which is chosen to be
compatible with the conventions in Section~\ref{sec-period}.

Next, when $\lambda$ take values in negative real numbers,
\begin{equation}\label{-2}
\begin{aligned}
&\frac{\partial\omega_1(\lambda)}{\partial\lambda}
=\int_{0}^{1}\sqrt{\frac{1}{t(1-t)(1-t\lambda)}}
\cdot\frac{t}{1-t\lambda}\mathrm{d}t, \\
&\frac{\partial\Phi_1(\lambda)}{\partial\lambda}
=\int_{0}^{1}\sqrt{\frac{t}{(1-t)(1-t\lambda)}}
\cdot\left(2+\frac{\lambda t}{1-t\lambda}\right)\mathrm{d}t,\\
&\frac{\partial\omega_2(\lambda)}{\partial\lambda}
=\mathrm{i}\int_{0}^{1}\sqrt{\frac{1}{x(1-x)(x-\lambda)}}
\cdot\frac{1}{x-\lambda}\mathrm{d}x,\\
&\frac{\partial\Phi_2(\lambda)}{\partial\lambda}
=\mathrm{i}\int_{0}^{1}\sqrt{\frac{x}{(1-x)(x-\lambda)}}
\cdot\frac{1}{x-\lambda}\mathrm{d}x.
\end{aligned}
\end{equation}
In particular, when $\lambda=-1$,
$\mathrm{Im}\omega_1=\mathrm{Im}\Phi_1=0,$
$\mathrm{Re}\omega_2=\mathrm{Re}\Phi_2=0.$ So $D_1=C_2=0$,
and the Jacobian matrix above is simplified to
\begin{align*}
\left(\frac{\partial\mathcal{P}_j}{\partial v_i}\right)_{v^*}=2
\begin{pmatrix}\begin{smallmatrix}
0&0&0&0&\frac{2\rho}{3}\mathrm{Re}\omega_1&0&0&\frac{2\rho}{3}\mathrm{Im}\omega_2\\
0&0&0&0&0&\frac{-2\rho}{3}\mathrm{Im}\omega_2&\frac{2\rho}{3}\mathrm{Re}\omega_1&0\\
\mathrm{Re}\Phi_1&0&0&0&0&0&0&0\\
0&-\mathrm{Im}\Phi_2&0&0&0&0&0&0\\
0&0&0&0&\frac{-2\rho}{3}\mathrm{Re}\omega_1&0&0&\frac{-2\rho}{3}\mathrm{Im}\omega_2\\
0&0&0&0&0&\frac{2\rho}{3}\mathrm{Im}\omega_2&\frac{-2\rho}{3}\mathrm{Re}\omega_1&0\\
0&0&\mathrm{Re}\Phi_1&0&0&0&0&0\\
0&0&0&-\mathrm{Im}\Phi_2&0&0&0&0\\
-\rho \mathrm{Re}\omega_1&0&-\rho \mathrm{Re}\omega_1&0&0&0&0&0\\
0&\rho \mathrm{Im}\omega_2&0&\rho \mathrm{Im}\omega_2&0&0&0&0\\
0&0&0&0&C_1+e_1&0&0&-D_2-f_2\\
0&0&0&0&0&D_2-f_2&C_1-e_1&0
\end{smallmatrix}\end{pmatrix}~.
\end{align*}
\textbf{Claim:} $\mathrm{rank}\left(\frac{\partial\mathcal{P}_j}{\partial v_i}\right)_{v^*}=8.$\\

To verify this claim, we need only to show
\begin{align*}
&\det\begin{pmatrix}
\frac{2\rho}{3}\mathrm{Re}\omega_1&0&0&\frac{2\rho}{3}\mathrm{Im}\omega_2\\
0&\frac{-2\rho}{3}\mathrm{Im}\omega_2&\frac{2\rho}{3}\mathrm{Re}\omega_1&0\\
C_1+e_1&0&0&-D_2-f_2\\
0&D_2-f_2&C_1-e_1&0\end{pmatrix}\\ =&-\left(\frac{2\rho}{3}\right)^2
\cdot(\mathrm{Re}\omega_1)^2\cdot\big[(C_1+D_2)^2-(e_1+f_1)^2\big]\ne 0.
\end{align*}
When $\lambda=-1$, By \eqref{-0}\eqref{-2}
they are expressed in terms of elliptic integrals,
\begin{align*}
-\Phi_1=\frac{2\rho^2}{3}\omega_1,~~\text{i.e.,}~
&\int_0^1\sqrt{\frac{t}{(1-t^2)}}\mathrm{d}t
=\frac{2\rho^2}{3}
\int_0^1\sqrt{\frac{1}{t(1-t^2)}}\mathrm{d}t;\\
C_1+D_2=-2\frac{\rho^2}{3}&\int_0^1
\sqrt{\frac{1}{t(1-t^2)}}\left(2t+\frac{1-t}{1+t}\right)
\mathrm{d}t,\\
e_1+f_2=&\int_0^1\sqrt{\frac{1}{t(1-t^2)}}
\left(2t+\frac{1-t}{1+t}t\right)\mathrm{d}t.
\end{align*}
We now prove $|C_1+D_2|<|e_1+f_2|$. This follows from
\begin{multline*}
\int_{0}^{1}\sqrt{\frac{t}{1-t^2}}\mathrm{d}t
\cdot\int_{0}^{1}\sqrt{\frac{1}{t(1-t^2)}}
\left(2t+\frac{1-t}{1+t}\right)\mathrm{d}t < \\
\qquad\qquad\qquad \int_{0}^{1}\sqrt{\frac{1}{t(1-t^2)}}\mathrm{d}t \cdot\int_0^1\sqrt{\frac{1}{t(1-t^2)}}
\left(2t+\frac{1-t}{1+t}t\right)\mathrm{d}t.
\end{multline*}
In fact, the left hand side is equal to the double integral
\[
\iint_{[0,1]\times[0,1]}\sqrt{\frac{1}{t(1-t^2)}}
\sqrt{\frac{1}{x(1-x^2)}}\left(2x+\frac{1-x}{1+x}\right)t \mathrm{d}t\mathrm{d}x;
\]
the right hand side is equal to the double integral
\[
\iint_{[0,1]\times[0,1]}\sqrt{\frac{1}{x(1-x^2)}}
\sqrt{\frac{1}{t(1-t^2)}}\left(2+\frac{1-t}{1+t}\right)t
\mathrm{d}t\mathrm{d}x.
\]
When $(x, t)\in [0,1]\times[0,1]$,
\[
1\leq 2x+\frac{1-x}{1+x} \leq 2 \leq 2+\frac{1-t}{1+t}\leq 3.
\]
So we obtain the desired inequality $|C_1+D_2|<|e_1+f_2|$.
The claim follows. This shows that around $v^*$, the preimage
of $\vec{0}$ is locally a 4-dimensional submanifold.

\begin{theorem}\label{thm-exist}
There exists a (real) 4-parameter family of deformations of
the Chen-Gackstatter surface which still has genus one,
a unique end, with total curvature $8\pi$, and
the Gauss maps $\phi,\psi$ both have order $1$ at the end.
\end{theorem}
\begin{proof}
According to the discussion of Case~1 in
Subsection~\ref{sub-case1} and the claim above,
there are a real 4-parameter family solutions to
the period conditions \eqref{period11}$\sim$\eqref{period13}
near $v^*$. They are obviously deformations of the
Chen-Gackstatter surface. We are left to verify the regularity condition when $v=(b,c,l,m,y_0,\lambda)$ is sufficiently close to
$v^*=(\rho,0,-\rho,0,0,-1)$.

When $v=v^*$, the corresponding surface is
the regular Chen-Gackstatter surface in $\mathbb{R}^3$.
We expect that under a small perturbation, the inequality
$\phi\ne\bar\psi$ still holds true. This is proved as below.

Taking $\lambda\in\hat{\mathbb{C}}-\{0,1,\infty\}$; in particular we choose $\lambda$ sufficiently close to $-1$.
Recall that on the elliptic curve
\begin{equation}\label{T1}
T_{\lambda} =\{[x,y,1]\in \mathbb{C}P^2
\mid y^2=x(x-1)(x-\lambda )\}\cup\{P=[0,1,0]\},
\end{equation}
the Gauss maps $\phi,\psi$ are given by \eqref{15}\eqref{16}:
\begin{equation}\label{T2}
\begin{aligned}
&\phi=\phi([x,y,1],v)=\frac{x-x_0}{b(y-y_0)+c(x-x_0)},\\
&\psi=\psi([x,y,1],v)=\frac{l(y+y_0)+m(x-x_0)}{x-x_0},
\end{aligned}
\end{equation}
where $x_0$ is an implicit function $x_0=x_0(\lambda, y_0)$
coming from \eqref{61},i.e.:
\begin{equation}\label{T3}
y_0^2=x_0(x_0-1)(x_0-\lambda ).
\end{equation}
Let $U$ be a neighborhood of $v^*$ in $\mathbb{C}^6$
on which $x_0=x_0(\lambda, y_0)$ is defined and differentiable.
We also identify the union of the 1-parameter family of tori $T^2_{\lambda}$
with the Cartesian product $T^2\times (-1-\delta,-1+\delta)$, using a topological trivialization which might be chosen quite
arbitrarily.
One expects that the formulas \eqref{T2} can define a continuous mapping
\[
\Xi:T^2\times U \rightarrow \widehat{\mathbb{C}}\times\widehat{\mathbb{C}},\qquad
([x,y,1],v)\mapsto(\phi,\bar\psi).
\]
At $P$ there is no trouble, because $\phi,\psi$
tends to $0$ and $\infty$ respectively as we know. Yet one should take care
when $(x,y)\to(x_0,\pm y_0)$. Observe
\[\phi=\frac{1}{b\frac{y-y_0}{x-x_0}+c}~,
\]
and $(b,c)$ is close to $(\rho,0)$. Using \eqref{T1}\eqref{T3} we derive
\[
y^2-y_0^2=(x-x_0)(x^2+xx_0+x_0^2-(\lambda+1)(x+x_0)+\lambda).
\]
Thus when $(x,y)\to(x_0,y_0)$, we have
\begin{displaymath}
\frac{y-y_0}{x-x_0}
=\left\{
\begin{array}{ll}
\frac{3x_0^2-2(\lambda+1)x_0+\lambda}{2y_0}
~~~&\text{when}~y_0\ne 0;\\
\frac{x^2-(\lambda+1)x+\lambda}{y}\to \infty,
~~~&\text{when}~y_0=0,\lambda\thicksim -1.
\end{array}
\right.
\end{displaymath}
From this we know $\phi$, as well as $\Xi$, is
well-defined at $[x,y,1]=[x_0,y_0,1]$. At $[x_0,-y_0,1]$ it is similar. In conclusion, $\Xi$ is well-defined and continuous.

Next we identify $\widehat{\mathbb{C}}$ with the unit sphere $S^2$ via the usual stereographic projection and consider
the standard distance function
\[
d:\widehat{\mathbb{C}}\times\widehat{\mathbb{C}}
\longrightarrow \mathbb{R}.
\]
Obviously, the composition
\[
\mathrm{dist}\triangleq d\circ\Xi: T^2\times U \longrightarrow \mathbb{R},
\]
is a continuous function on $T^2\times U\subset
\mathbb{C}P^2\times \mathbb{C}^6$.
Moreover, it is uniform continuous with respect to
the variable $[x,y,1]\in T^2$.

The final observation is
\[
\mathrm{dist}\mid_{T^2\times \{v^*\}}=\pi/2,
\]
because in that case $\phi=-1/\psi$, and $\{\phi,\bar\psi\}$ always
form a pair of antipodal points on $S^2=\widehat{\mathbb{C}}$.
Thus there exists a smaller neighborhood $V$ of $v^*$ such that \[
\mathrm{dist}\mid_{T^2\times V}>0.
\]
This implies, when $(b,c,l,m,y_0,\lambda)\in V$,
there is always $\phi\ne\bar\psi$ on the whole torus $T^2\subset\mathbb{C}P^2$. Hence the small deformation near $v^*$ yields complete immersed stationary surface as we desired.
\end{proof}

\section{A uniqueness result when there are many symmetries}
\label{sec-unique}

This section is devoted to the proof of a uniqueness result when the surface
in consideration has many symmetries. Here, a symmetry of a surface
$X:M\to\mathbb{R}^4_1$ means a Lorentz transformation $\mathbbm{g}$ preserving the non-parameterized surface, i.e.,
\[
g\in \mathrm{Isom}(\mathbb{R}^4_1),~~g(X(M))=X(M).
\]
Note that we have shown in Section~\ref{sec-deform} that
the Lorentz deformation preserves the symmetry group of
the original Chen-Gackstatter surface.

\begin{theorem}\label{thm-unique}
Suppose $X:M=T^2-\{P\} \to \mathbb{R}_1^4$ is a
complete regular algebraic stationary surface of
total curvature $-\int_M K \mathrm{d}M=8\pi$.
If it has a symmetry group $\mathbbm{G}$ with more than 4 elements, then this surface must be the classical Chen-Gackstatter surface in $\mathbb{R}^3$ or one of its Lorentz deformations.
\end{theorem}
Any symmetry $\mathbbm{g}$ induces a conformal automorphism of the underlying
torus. Under the assumption $|\mathbbm{G}|>4$, there are also more than
$4$ conformal automorphisms for this $T^2$.
They have a common fixed point, the end $P$.
Such a torus must be conformally equivalent to
either a square torus, or a equilateral torus.
(This assertion is a folklore; one may consult \cite{Eujalance}.)
The basic information for these two types are as below. (Denote $\epsilon\triangleq\mathrm{e}^{\pi\mathrm{i}/3}$ with $\epsilon^3=-1$.)
\begin{figure}[htbp]
\centering
\subfigure[$\tau=\mathrm{i}$]{
\label{square}
\includegraphics[width=0.4\textwidth]{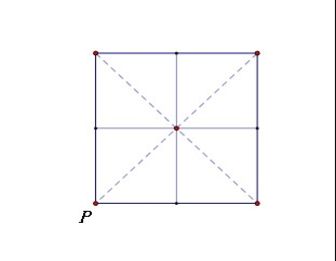}}
\subfigure[$\tau=\mathrm{e}^{\pi\mathrm{i}/3}$]{
\label{rhombic}
\includegraphics[width=0.5\textwidth]{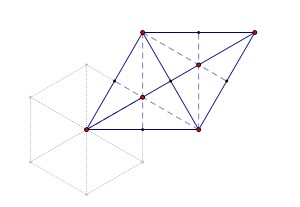}}
\caption{Fundamental domain of symmetric torus}
\label{fundamental}
\end{figure}
\begin{itemize}
\item The conformal automorphism group for the square torus
(or the equilateral torus), including anti-holomorphic transformations, is the dihedral group $D_4$ (or $D_6$), with $8$ (or $12$) elements. Especially it is a finite group as well as $G$.
\item The fundamental domain $\Omega$ is a square (or a
rhombic with one angle equal to $\pi/3$);
\item This torus is equivalent to
$\mathbb{C}/\Lambda(1,\tau)$ with period quotient
$\tau=\mathrm{i}$ (or $\tau=\epsilon$);
\item They can be represented by the elliptic curve
$y^2=x(x-1)(x-\lambda)$ with modular invariant $\lambda=-1$ (or $\lambda=\epsilon$).
\end{itemize}

To examine the symmetry of a Chen-Gackstatter type surface $M$
in $\mathbb{R}^4_1$, first note the symmetry group $\mathbbm{G}$ is
finite; it contains no element of translation in $\mathbb{R}^4_1$.
Then recall that in both Case~1 and 2 in
Section~\ref{sec-divisor} we have chosen and fixed some
coordinate system of $\mathbb{R}^4_1$ so that
\[
\phi(P)=0,\psi(P)=\infty.
\]
That means the limit of the normal plane at $P$ is $X_1=X_2=0$,
spanned by $(0,0,1,1),(0,0,-1,1)$.
So any symmetry $\mathbbm{g}$ of $M$ must preserve the spacelike 2-plane
$X_3=X_4=0$ and the Lorentz 2-plane $X_1=X_2=0$ because
this corresponds to the tangent and normal space decomposition at the end $P$.
We conclude that $\mathbbm{g}\in O(2)\times O(1,1)$.

Consider the action of $\mathbbm{g}$ on the Weierstrass data.
By the Weierstrass representation formula
\[
X_z\mathrm{d}z=(\Theta_1,\Theta_2,\Theta_3,\Theta_4)
=\Big(\phi+\psi, -i(\phi-\psi),
1-\phi\psi,1+\phi\psi\Big)\mathrm{d}h,
\]
we have
\[
\phi\mathrm{d}h=\frac{\Theta_1+\mathrm{i}\Theta_2}{2}~,~
\psi\mathrm{d}h=\frac{\Theta_1-\mathrm{i}\Theta_2}{2}~,~
\mathrm{d}h=\frac{\Theta_3+\Theta_4}{2}~,~
\phi\psi\mathrm{d}h=\frac{\Theta_3-\Theta_4}{2}~.
\]
When $\mathbbm{g}$ induces a holomorphic automorphism of $M$,
the pull back of these meromorphic 1-forms via
$\mathbbm{g}\in O(2)\times O(1,1)$ are
\begin{equation}\label{71}
\left\{\begin{aligned}
\mathbbm{g}^*(\phi\mathrm{d}h)&=\mathrm{e}^{\mathrm{i}\theta}\phi\mathrm{d}h,\\
\mathbbm{g}^*(\psi\mathrm{d}h)&=\mathrm{e}^{-\mathrm{i}\theta}\psi\mathrm{d}h,
\end{aligned}\right.
~~~~\text{or}~~
\left\{\begin{aligned}
\mathbbm{g}^*(\phi\mathrm{d}h)&=\mathrm{e}^{\mathrm{i}\theta}\psi\mathrm{d}h,\\
\mathbbm{g}^*(\psi\mathrm{d}h)&=\mathrm{e}^{-\mathrm{i}\theta}\phi\mathrm{d}h;
\end{aligned}\right.
\end{equation}
and
\begin{equation}\label{72}
\left\{\begin{aligned}
\mathbbm{g}^*(\mathrm{d}h)~~~~&=\mu\mathrm{d}h,\\
\mathbbm{g}^*(\phi\psi\mathrm{d}h)&=\frac{1}{\mu}\phi\psi\mathrm{d}h,
\end{aligned}\right.
~~~~\text{or}~~
\left\{\begin{aligned}
\mathbbm{g}^*(\mathrm{d}h)~~~~&=\mu\phi\psi\mathrm{d}h,\\
\mathbbm{g}^*(\phi\psi\mathrm{d}h)&=\frac{1}{\mu}\mathrm{d}h.
\end{aligned}\right.
\end{equation}
If $\mathbbm{g}$ reverses the orientation of $M$ (i.e. anti-holomorphic
automorphism), the pull back of these meromorphic 1-forms
are \eqref{71}\eqref{72} composed with a complex conjugation.
(Since $\mathbbm{G}$ is a finite group, we know $\mu=\pm 1$, and $\theta/\pi$ is some rational number.
Yet this fact will not be used later.)

We point out that in either Case~1 or Case~2, according to
the information derived in Section~\ref{sec-divisor},
$\phi\mathrm{d}h$ and $\psi\mathrm{d}h$ have poles of
different orders at the fixed point $P$, whereas
$\mathrm{d}h$ and $\phi\psi\mathrm{d}h$ have poles of equal orders.
Since the pull back via $\mathbbm{g}$ will not change the order of pole
at $P$, in \eqref{71} we actually have
$\mathbbm{g}^*(\phi\mathrm{d}h)=\mathrm{e}^{\mathrm{i}\theta}
\phi\mathrm{d}h$ (or post-composed with a complex conjugation).
For $\mathbbm{g}^*(\psi\mathrm{d}h)$ the conclusion is similar.
Especially, we have the following important conclusion.

\begin{lemma}\label{lem-unique}
Assumptions as in Theorem~\ref{thm-unique}.
Then the divisors $(\phi\mathrm{d}h)$ and $(\psi\mathrm{d}h)$
are preserved by any symmetry $\mathbbm{g}$ in both Case~1 and 2,
while $(\mathrm{d}h),(\phi\psi\mathrm{d}h)$
are preserved or interchanged between each other.
\end{lemma}

Now we prove Theorem~\ref{thm-unique} through a case-by-case analysis. \\

\noindent
\textbf{Case~1.} Recall that in Subsection~\ref{sub-case1}
we obtained the divisors:
\begin{align}
&(\phi)=P+C_1-V_1-V_2,~ (\psi)=-P-C_2+N_1+N_2; \label{73}\\
&(\phi \mathrm{d}h)=-2P+C_1+C_2,~
(\psi \mathrm{d}h)=-4P+V_1+V_2+N_1+N_2; \label{74}\\
&(\mathrm{d}h)=-3P+V_1+V_2+C_2,~
(\phi \psi \mathrm{d}h)=-3P+C_1+N_1+N_2. \label{75}
\end{align}
By Lemma~\ref{lem-unique}, we know $\{C_1,C_2\}$, the zero set
of $\phi\mathrm{d}h$, must be preserved by the symmetry group $\mathbbm{G}$. Now we discuss the two conformal types separately.\\

\textbf{(1.1) Square torus.} $\mathbbm{G}$ is a subgroup
of $D_4$ with more than $4$ elements, so $\mathbbm{G}=D_4$, which
contains all the symmetries of the square (as the fundamental domain $\Omega$).
Since $\{C_1,C_2\}$ is an invariant subset under $\mathbbm{G}=D_4$,
the only possibility is $C_1=C_2$ located at the center of $\Omega$.
This time $\phi\mathrm{d}h=(x-x_0)\mathrm{d}z$ has
a double zero; so we may take $x_0=0$ in \eqref{eq-x0}.
By Proposition~\ref{prop-unique}, such surfaces are Lorentz
deformations of the classical Chen-Gackstatter surface.\\

\textbf{(1.2) Equilateral torus.}  $\mathbbm{G}$ might be $D_6$
or one of its subgroups $D_3, Z_6$. In any case $\mathbbm{G}$ contains the
subgroup $Z_3$. Take the fundamental domain $\Omega$ as
the rhombic below, which is the union of two equilateral triangles.
Its subset $\{C_1,C_2\}$, invariant under the action of
this $Z_3$,
obviously consists of the centers of these two triangles.
Except these three fixed points $C_1,C_2,P$, every orbit of
$Z_3$ on $T^2$ contains $3$ distinct points.

On the other hand, the zero set of $\psi\mathrm{d}h$
is $\{V_1,V_2,N_1,N_2\}$. Among them, $\{V_1,V_2\}$ and
$\{N_1,N_2\}$ are invariant or interchanged under the action of
$Z_3$, because they are zeros of $\mathrm{d}h$
and $\phi\psi\mathrm{d}h$, respectively.
This forces $V_1=V_2=C_2, N_1=N_2=C_1$.
(Note that $V_1=V_2$ can not be $C_1$ since they are respectively
the pole and a zero of $\phi$. Similarly, $N_1=N_2\ne C_2$.)
But this time $\phi(C_1)=\psi(C_1)=0,\phi(C_2)=\psi(C_2)=\infty$. This contradicts with the regularity condition $\phi\ne\bar\psi$. (A little bit more analysis will show that the period conditions are also violated.)
So this subcase is ruled out.\\

\noindent
\textbf{Case~2.} We discuss the two conformal types separately.\\

\textbf{(2.1) Square torus.} We will prove a stronger result:
In Case~2 there is no example defined on the square torus.
Otherwise, suppose there is such a surface with $\lambda=-1$.
Invoking the result in Subsection~\ref{sub-case2} and
put $e=1,-ax_0=a',\lambda=-1$ in \eqref{22}$\sim$\eqref{27}, we get the Weierstrass data
\[
\phi=1/(ax+a'),~~ \psi=bx+c,~~\mathrm{d}h=(ax+a')\mathrm{d}z;
\]
and the period conditions
\begin{align}
&\mathrm{Re}[a\Phi_i+a'\omega_i]=0, \label{eq-period71}\\
&\mathrm{Re}[b\Phi_i+c\omega_i]=0, \label{eq-period72}\\
&-\overline{\omega_i}
=(ac+a'b)\Phi_i +
\Big(a'c+\frac{ab}{3}\Big)\omega_i. \label{eq-period73}
\end{align}
It is known (see \eqref{-1}) that when $\lambda=-1$,
$\Phi_2=-\mathrm{i}\Phi_1\in\mathrm{i}\mathbb{R}$,
$\omega_2=\mathrm{i}\omega_1\in\mathrm{i}\mathbb{R}$.
Insert this into \eqref{eq-period71}\eqref{eq-period72},
we obtain
\[
\overline{a'}/a=\bar{c}/b=-\Phi_1/\omega_1=2\rho^2/3.
\]
Put this back to \eqref{eq-period73} and set $i=1,2$.
It follows
\[
1=(\bar{a}b+a\bar{b})\frac{4}{9}\rho^4;~~~
\bar{a}\bar{b}\cdot\frac{4}{9}\rho^4+\frac{1}{3}ab=0.
\]
The second equality implies $\rho^4=3/4=0.75$.
Yet this contradicts with what we knew before in \eqref{eq-rho} that
$\rho\approx 0.8279$.\\

\textbf{(2.2) Equilateral torus.} As in previous discussion in (1.2), $\mathbb{G}$ contains the subgroup $Z_3$.
In Subsection~\ref{sub-case2}
we obtained the divisors:
\begin{gather}
(\phi)=2P-V_1-V_2,~ (\psi)=-2P+C_1+C_2,~ (\mathrm{d}h)=-2P+V_1+V_2. \label{77}\\
(\phi\mathrm{d}h)=0,(\psi\mathrm{d}h)=
-4P\!+\!C_1\!+\!C_2\!+\!V_1\!+\!V_2,
(\phi \psi \mathrm{d}h)=-2P\!+\!C_1\!+\!C_2. \label{78}
\end{gather}
By Lemma~\ref{lem-unique}, we know $\{C_1,C_2\}$, the zero set
of $\phi\psi\mathrm{d}h$, and $\{V_1,V_2\}$, the zero set
of $\mathrm{d}h$, must keep invariant under $Z_3\subset \mathbb{G}$,
or interchange with each other. Since the sum of either pair
on the fundamental domain $\Omega\subset\mathbb{C}$
must be in the lattice, the only possibility is that both
two pairs coincide with the centers of the two equilateral
triangles (the fixed points of $Z_3$ other than $P$).
This implies that $(\phi\psi)=0$, hence $\phi\psi$ is a constant.
By Theorem~\ref{thm-deform}, this means this surface comes
from the Lorentz deformation of a Chen-Gackstatter type
surface $M'$ in $\mathbb{R}^3$ which is conformal to a
equilateral torus punctured at one point. But according to the
uniqueness result of Lopez \cite{Lopez-1992}, there does not
exist such a surface $M'$ in $\mathbb{R}^3$. This completes the proof to the whole uniqueness theorem stated at the beginning.
(The final argument can be replaced by using Kusner's observation that such a surface $M'\subset \mathbb{R}^3$ must have a triple point, which violates the monotonicity formula in $\mathbb{R}^3$. See the next section.)

\section{Further discussion on existence and embedding problems}

In conclusion to our discussions in the preceding sections,
it is difficult to give a complete answer to Question~1 and 2 posed in the introduction, i.e., the problem of existence, uniqueness, and embeddedness. In particular, it is not known whether there exist
any example in $\mathbb{R}^4_1$ with Weierstrass data given in Case~2.
Note that in $\mathbb{R}^3$ such examples do not exist.
But we can prove the following non-embedding result.
\begin{theorem}\label{thm-triple}
Suppose there exists a complete regular algebraic stationary surface
$X: T^2-\{P\} \to \mathbb{R}^4_1$ with total curvature $-\int_M K \mathrm{d}M=8\pi$, whose Gauss map $\{\phi, \psi\}$ have order 2 at the end $P$(Case~2). Then this surface is not embedded.
\end{theorem}
\begin{proof}
Follow the discussion in Subsection~\ref{sub-case2}, we consider
an involution of $T^2:\{[x,y,1]|y^2=x(x-1)(x-\lambda)\}\cup\{[0,1,0]\}$ as below:
\begin{align*}
I:  \qquad T^2 &\longrightarrow \quad T^2\\
    [x,y,1]&\longmapsto  [x,-y,1]
\end{align*}
From \eqref{22}\eqref{24}, $\phi=1/(ax-ax_0),\psi=bx+c,\mathrm{d}h=a(x-x_0)\mathrm{d}x/y,$
we get
\[
\phi\circ I=\phi,~\psi\circ I=\psi,~I^*(\mathrm{d}h)=-\mathrm{d}h.
\]
So $I^*(X_z\mathrm{d}z)=-X_z\mathrm{d}z$. Taking a translation in $\mathbb{R}_1^4$ if necessary, $I$ induces a transformation on $X(M)$, which is the restriction of the central symmetry $\widetilde{I}:X\to -X$ with a unique fixed point $(0,0,0,0)$ in $\mathbb{R}_1^4$.
On the other hand, $(x,y)=(0,0),(1,0),(\lambda,0)$ are three fixed points of the transformation $I$ on this $T^2-\{P\}$. So $X(P_i)$ are three fixed points under the transformation
$\widetilde{I}$. We conclude $X(P_i)=(0,0,0,0)$. This is a triple
point. So $X(M)$ can not be embedded.
\end{proof}

In $\mathbb{R}^3$, the argument
above implies a contradiction with the monotonicity formula,
which shows the non-existence of such examples in Case~2.

In view of this, we consider the problem of generalizing
the monotonicity formula to $\mathbb{R}_{1}^{4}$.
Since the metric of $\mathbb{R}_{1}^{4}$ and its restriction
on the normal bundle are indefinite,
the direct generalization seems not true.
On the other hand, so far we did not find counter-examples to the multiplicity inequality mentioned in the introduction.
Thus we make the following\\

\textbf{Conjecture (the multiplicity inequality)}: For a complete, immersed, algebraic stationary surface ${\bf x}:M\to \mathbb{R}_{1}^{4}$ with $r$ regular ends, suppose it has multiplicity $m$ at one point ${\bf x}(q)$, and its multiplicity at the end $p_j$ ($1\le j\le r$ is $d_j$. Then $m<\sum_{j=1}^r d_j$.\\

For complete stationary surfaces with total curvature
$-\int K\mathrm{d}M=8\pi$ and embedded in $\mathbb{R}^4_1$,
in \cite{Liu-Ma-Wang-Wang-2012} we have shown that
the generalized $3$-noid could be such an example.
So far it is still an open question whether there exist genus one examples. By Theorem~\ref{thm-triple}, any of such examples
with a unique end must be of Case~1.
Yet among the 4-parameter family of deformations known to us,
except the 2-parameter family of Lorentz deformations
(Proposition~\ref{prop-deform}),
for any of these Chen-Gackstatter type surfaces we do not know whether it is embedded.

In contrast, in $\mathbb{R}^4$ the same problem is easy to answer:
there are many embedded minimal tori with one end and total
Gauss curvature $-\int K\mathrm{d}M=8\pi$. The simplest construction is the elliptic curves with one puncture
\[
T^2_{\lambda}-\{P\}=\{(x,y)|y^2=x(x-1)(x-\lambda)\}\subset \mathbb{C}^2=\mathbb{R}^4.
\]
It is easy to see that it is embedded with a unique end of multiplicity $3$. By the Jorge-Meeks formula it has total curvature $8\pi$.

For possible examples of genus one with two catenoid ends,
we will consider it in the future. Note that if there exists
such a stationary surface in $\mathbb{R}^4_1$, when it has self-intersection, we get a
counter-example to the conjecture above; otherwise we will
have a desired embedded stationary surface. Either conclusion
is interesting.

\vspace{5mm} \noindent Zhenxiao Xie, {\small\it School of
Mathematical Sciences and Beijing International Center for Mathematical Research, Peking University, 100871 Beijing, People's
Republic of China. e-mail: {\sf xiezhenxiao@126.com}}

\vspace{5mm} \noindent Xiang Ma, {\small\it LMAM, School of
Mathematical Sciences, Peking University, 100871 Beijing, People's
Republic of China. e-mail: {\sf maxiang@math.pku.edu.cn}}

\end{document}